\documentclass[11pt]{article}

\usepackage[T1]{fontenc}
\usepackage{lmodern}
\usepackage[margin=1in]{geometry}
\usepackage{amsmath,amssymb,amsthm,mathtools}
\usepackage{enumitem}
\usepackage{xcolor}
\usepackage{microtype}
\usepackage{authblk}
\usepackage[colorlinks=true,linkcolor=blue,citecolor=blue,urlcolor=blue]{hyperref}
\usepackage[capitalise]{cleveref}

\numberwithin{equation}{section}
\setlist{nosep}
\setlist[enumerate]{label=\textup{(\arabic*)}}

\newtheorem{theorem}{Theorem}[section]
\newtheorem{lemma}[theorem]{Lemma}
\newtheorem{proposition}[theorem]{Proposition}
\newtheorem{corollary}[theorem]{Corollary}
\newtheorem{claim}[theorem]{Claim}

\newtheorem{problem}[theorem]{Problem}
\theoremstyle{remark}

\newcommand{\N}{\mathbb N}
\newcommand{\cR}{\mathcal R}
\newcommand{\cP}{\mathcal P}
\newcommand{\Ent}{\operatorname{Ent}}
\newcommand{\eps}{\varepsilon}

\title{Entropy Transference for Rainbow-$H$-Free Colourings of \\Random Graphs}

\author[1]{Mengyu Cao\thanks{E-mail: \texttt{myucao@ruc.edu.cn}. Supported by the National Natural Science Foundation of China (12301431) and Beijing Natural Science Foundation (1262010).}}
\author[2]{Mei Lu\thanks{E-mail: \texttt{lumei@tsinghua.edu.cn}. M. Lu is supported by the National Natural Science Foundation of China (Grant 12571372) and Beijing Natural Science Foundation (Grant 1262010).}}
\author[2]{Haixiang Zhang\thanks{Corresponding author. E-mail: \texttt{zhang-hx22@mails.tsinghua.edu.cn}.}}

\affil[1]{\small Institute for Mathematical Sciences, Renmin University of China, Beijing 100086, China}
\affil[2]{\small Department of Mathematical Sciences, Tsinghua University, Beijing 100084, China}

\date{}

\hypersetup{
	pdftitle={Entropy Transference for Rainbow-H-Free Colourings of Random Graphs},
	pdfauthor={Mengyu Cao, Mei Lu and Haixiang Zhang},
	pdfkeywords={rainbow subgraph, random graph, sparse regularity, template entropy, stability}
}

\begin{document}
	\maketitle
	
\begin{abstract}
		Let $H$ be a fixed graph with $e(H)\ge3$ that contains two adjacent edges, and let $\ell\ge e(H)$ be fixed.  We establish an entropy-transference principle for rainbow-$H$-free edge-colourings of the binomial random graph at the natural scale $p=n^{-1/m_2(H)}$.  Writing $R_{H,\ell}(G)$ for the number of such colourings and $\lambda(H,\ell)$ for the rainbow entropy--Tur\'an density on complete graphs, we show that, with high probability, the per-edge logarithmic counting rate can be made arbitrarily close to $\log\ell$ below a sufficiently small constant multiple of this scale, and arbitrarily close to $\lambda(H,\ell)$ above a sufficiently large constant multiple.  Thus the dense-side counting rate on a sparse random host is governed exactly by a deterministic entropy--Tur\'an parameter on complete graphs.  We further investigate this parameter, obtaining partial exact evaluations, corresponding counting-stability results, and its first-order asymptotic behaviour as the number of colours tends to infinity.  This extends the random Gallai-colouring transition from triangles to every fixed non-matching graph containing at least three edges, and provides a general mechanism for transferring complete-graph template entropy to sparse random hosts.
	\end{abstract}
	
	
	
	\section{Introduction}
	
	\subsection{Colourings without rainbow subgraphs}
	
	A copy of a graph $H$ in an edge-coloured graph is \emph{rainbow} if its edges receive pairwise distinct colours.  Graph copies are not required to be induced.  The systematic study of colourings without rainbow subgraphs began with the anti-Ramsey theory of Erd\H{o}s, Simonovits and S\'os \cite{ErdosSimonovitsSos}.  Classical anti-Ramsey questions ask how many colours can be used on a complete graph while avoiding a rainbow copy of a prescribed graph; structural and local variants have also been studied, for example under restrictions on the multiplicity of each colour at a vertex \cite{AlonJiangMillerPritikin}.  Our problem is enumerative: the number of available colours is fixed, and we count all colourings containing no rainbow copy of $H$.
	
	For a positive integer $s$, write $[s]=\{1,\ldots,s\}$.  For a graph $G$ and a positive integer $\ell$, let
	\[
	\cR_{H,\ell}(G)
	:=\{\chi:E(G)\to[\ell]:\chi\text{ contains no rainbow copy of }H\},
	\qquad
	R_{H,\ell}(G):=|\cR_{H,\ell}(G)|.
	\]
	The basic model case is $H=K_3$.  Colourings with no rainbow triangle are called \emph{Gallai colourings}.  Gallai's theorem states that every such colouring of $K_n$, $n\ge2$, admits a partition into at least two nonempty vertex sets such that all edges between any fixed pair of parts have one colour and at most two colours occur between distinct parts \cite{Gallai,GyarfasSimonyi}.  This recursive structure has an enumerative consequence.  Independently, Balogh and Li \cite{BaloghLi} and Bastos, Benevides and Han \cite{BastosBenevidesHan} proved that, for every fixed $r\ge3$,
	\[
	R_{K_3,r}(K_n)
	=\left(\binom r2+o(1)\right)2^{\binom n2}.
	\]
	In particular, the proportion of Gallai $r$-colourings of $K_n$ that use only two colours tends to one, even though Gallai $r$-colourings using more than two colours do exist.  Thus a counting theorem may describe almost all admissible colourings without giving a pointwise classification of every admissible colouring.  Gallai colourings on non-complete graphs were studied earlier by Gy\'arf\'as and S\'ark\"ozy \cite{GyarfasSarkozy}.
	
	Counting edge-colourings with forbidden colour patterns is closely related to the Erd\H{o}s--Rothschild problem \cite{Erdos1974}.  For monochromatic forbidden cliques, Alon, Balogh, Keevash and Sudakov \cite{AlonBaloghKeevashSudakov} obtained exact enumeration results by combining regularity with extremal graph theory.  Prescribed colour patterns were subsequently studied by Benevides, Hoppen and Sampaio \cite{BenevidesHoppenSampaio}, while Hoppen, Lefmann and Odermann formulated and developed a rainbow Erd\H{o}s--Rothschild problem \cite{HoppenLefmannOdermann}.
	
	Colouring templates and their entropy provide a general language for complete-graph counting problems \cite{FROU,GuptaPehovaPowierskiStaden}.  A template assigns to each edge a nonempty set of permitted colours, called its palette.  If $\cP$ is a template, then exactly $\prod_e|\cP(e)|$ colourings realise it, and its natural-logarithmic entropy is
	\[
	\Ent(\cP)=\sum_e\log|\cP(e)|.
	\]
	Equivalently, this is the Shannon entropy of the product measure obtained by choosing independently and uniformly from each edge palette.  Specialising the general counting theorem of Falgas-Ravry, O'Connell and Uzzell \cite[Corollary~2.15]{FROU} to rainbow-$H$-free colourings gives
	\begin{equation}\label{eq:complete-host-count-intro}
	R_{H,\ell}(K_n)
	=\exp\left((\lambda(H,\ell)+o(1))\binom n2\right),
	\end{equation}
	where $\lambda(H,\ell)$ is the limiting maximum entropy per edge defined formally in \cref{sec:deterministic-entropy}.  They also determined the entropy and its extremal templates for $(H,\ell)=(K_3,3)$ and proved stability in that case \cite[Theorems~4.13 and~4.16]{FROU}.  Thus the existence and abstract entropy interpretation of $\lambda(H,\ell)$ are implicit in the general hereditary-property framework; one aim here is to obtain explicit graph-theoretic information about this parameter.
	
	In the rainbow Erd\H{o}s--Rothschild problem, one instead maximises $R_{H,\ell}(G)$ over all $n$-vertex graphs $G$.  H\`an, Hoppen, M\"uller and Schmidt determined the extremal graph for rainbow-$K_4$-free colourings when the number of colours is at least $12$ \cite{HanHoppenMullerSchmidt}.  In contrast, the host in \eqref{eq:complete-host-count-intro} is $K_n$, while our main problem takes the host to be the random graph $G(n,p)$.
	
	\subsection{Random graphs and entropy transference}
	
	Benevides, Monteiro and Mota initiated the random-host counting problem for Gallai colourings \cite{BMM2026}.  We write $G(n,p)$ for the binomial random graph on $[n]$, in which every edge appears independently with probability $p$.  Writing $\operatorname{gc}_3(G)=R_{K_3,3}(G)$, their theorem is as follows.
	
	\begin{theorem}[Benevides--Monteiro--Mota~\cite{BMM2026}]\label{thm:random-gallai}
		For every $\delta>0$ there are constants $a,A>0$ such that, with high probability,
		\[
		\begin{cases}
			\operatorname{gc}_3(G(n,p))
			\ge 3^{(1-\delta)e(G(n,p))},
			& p\le a n^{-1/2},\\[1ex]
			\operatorname{gc}_3(G(n,p))
			\le 2^{(1+\delta)e(G(n,p))},
			& p\ge A n^{-1/2}.
		\end{cases}
		\]
	\end{theorem}
	
	In the two ranges of the theorem, the edge-normalised logarithmic count changes from $\log 3$ to $\log 2$.  Below $n^{-1/2}=n^{-1/m_2(K_3)}$ by a sufficiently small constant factor, one can leave all but an arbitrarily small proportion of the edges unconstrained; above this scale by a sufficiently large constant factor, the leading exponential base is that of the two-colour construction.  The maximum $2$-density
	\[
	m_2(H):=\max_{\substack{F\subseteq H\\v(F)\ge3}}
	\frac{e(F)-1}{v(F)-2}
	\]
	is the parameter that governs the sparse embedding threshold used here.  A graph $F$ with $v(F)\ge3$ is \emph{$2$-balanced} if
	\[
	\frac{e(J)-1}{v(J)-2}
	\le \frac{e(F)-1}{v(F)-2}
	\qquad\text{for every }J\subseteq F\text{ with }v(J)\ge3.
	\]
	Thus every subgraph $F\subseteq H$ attaining $m_2(H)$ is $2$-balanced.  In particular, the Kohayakawa--{\L}uczak--R\"odl (K{\L}R) theorem embeds regular $H$-partite configurations in $G(n,p)$ at the scale $n^{-1/m_2(H)}$ \cite{CGSS,GerkeSteger}.  This suggests asking whether the Gallai transition extends from triangles to every fixed forbidden rainbow graph at the same scale.
	\iffalse
	Two difficulties arise beyond the three-colour triangle case.  First, if $\ell\ge e(H)$, one must control all possible Hall obstructions among the $e(H)$ palettes on each copy of $H$; several different sizes of obstructing subfamilies may compete.  Second, for large $\ell$, the template that assigns the same $(e(H)-1)$-set of colours to every edge need not attain the optimal limiting entropy density: templates based on multipartite graphs avoiding an edge-deleted copy of $H$ can use all colours on a positive proportion of the pairs.  The random transference must therefore be separated from the complete-graph entropy--Tur\'an problem.
	\fi
	
	Unless another limiting variable is displayed explicitly, all asymptotic statements are taken as $n\to\infty$, and ``with high probability'' means with probability tending to one.  Throughout, $H$ is a fixed simple graph; isolated vertices play no role and are deleted.  For $S\subseteq E(H)$, the notation $H-S$ denotes the spanning subgraph obtained by deleting the edges in $S$, and $H-e:=H-\{e\}$.  We assume that $e(H)\ge3$ and that $H$ is not a matching, equivalently that $H$ contains two adjacent edges.  Then $m_2(H)\ge1$, and we fix an integer $\ell\ge e(H)$.
	The terms \emph{below-threshold} and \emph{above-threshold} refer, respectively, to probabilities bounded by a sufficiently small constant multiple and bounded below by a sufficiently large constant multiple of $n^{-1/m_2(H)}$.
	
	We call a complete-graph template \emph{$H$-SDR-free} if the palettes on every copy of $H$ have no system of distinct representatives.  Thus the term refers only to templates; colourings themselves will always be called rainbow-$H$-free.  We call the limiting maximum entropy per edge over $H$-SDR-free templates the \emph{rainbow entropy--Tur\'an density} and denote it by $\lambda(H,\ell)$.  Its formal definition and the existence of the limit are given in \cref{sec:deterministic-entropy}.  The main theorem shows that this parameter is the edge-normalised exponential rate in the above-threshold random-graph regime.
	
	\begin{theorem}\label{thm:lambda-transfer}
		Let $H$ be a fixed non-matching graph with $e(H)\ge3$, and fix an integer $\ell\ge e(H)$.  For every $\delta>0$ there are constants $a=a(H,\ell,\delta)>0$ and $A=A(H,\ell,\delta)>0$ such that, for $G=G(n,p)$, with high probability
		\[
		\begin{cases}
			\ell^{(1-\delta)e(G)}\le R_{H,\ell}(G)\le\ell^{e(G)},
			& p\le a n^{-1/m_2(H)},\\[1ex]
			\exp\bigl((\lambda(H,\ell)-\delta)e(G)\bigr)
			\le R_{H,\ell}(G)
			\le \exp\bigl((\lambda(H,\ell)+\delta)e(G)\bigr),
			& p\ge A n^{-1/m_2(H)}.
		\end{cases}
		\]
	\end{theorem}
	
	Thus the theorem determines the edge-normalised exponential rate in the below-threshold and above-threshold ranges.  In the latter range, the rate is reduced to a deterministic complete-graph template problem with fixed $\ell$ that no longer involves $p$.  We refer to this identification of the random-graph rate with the complete-graph template density as \emph{entropy transference}.
	
	An exact multiplicative form of Hall's obstruction determines this density whenever
	$e(H)\le\ell\le(e(H)-1)^{e(H)/(e(H)-2)}$: the template that assigns one fixed $(e(H)-1)$-set to every edge is extremal.  We call this interval the \emph{Hall range}; it is universal in the sense that it depends on $H$ only through $e(H)$.  Combining this fact with \cref{thm:lambda-transfer} gives the following consequence.
	
	\begin{corollary}\label{cor:hall-transition}
		Let $H$ be a fixed non-matching graph with $e(H)\ge3$, and let $e(H)\le\ell\le(e(H)-1)^{e(H)/(e(H)-2)}$ be a fixed integer.  For every $\delta>0$ there are constants $a=a(H,\ell,\delta)>0$ and $A=A(H,\ell,\delta)>0$ such that, for $G=G(n,p)$, with high probability
		\[
		\begin{cases}
			\ell^{(1-\delta)e(G)}\le R_{H,\ell}(G)\le\ell^{e(G)},
			& p\le a n^{-1/m_2(H)},\\[1ex]
			(e(H)-1)^{e(G)}\le R_{H,\ell}(G)\le(e(H)-1)^{(1+\delta)e(G)},
			& p\ge A n^{-1/m_2(H)}.
		\end{cases}
		\]
	\end{corollary}
	
	At the transition scale, $e(G(n,p))=(1+o(1))p\binom n2$ with high probability, so the edge-normalised statements are also equivalent to formulations using $p\binom n2$.  Taking $\ell=e(H)$ gives the least nontrivial number of available colours, since no copy of $H$ can be rainbow when $\ell<e(H)$; taking $H=K_3$ and $\ell=3$ recovers \cref{thm:random-gallai}.  The same two-regime conclusion therefore holds for every non-matching $H$ and every fixed integer $\ell$ in the Hall range, with threshold scale $n^{-1/m_2(H)}$.
	
	Below the upper endpoint of the Hall range, the same template satisfies the following counting-stability statement.  This is not a pointwise classification: the theorem bounds the number of rainbow-$H$-free colourings that are far from every fixed $(e(H)-1)$-set of colours.  For $\chi:E(G)\to[\ell]$, define
	\[
	d_{e(H)-1}(\chi)
	:=\min_{B\in\binom{[\ell]}{e(H)-1}}
	|\{e\in E(G):\chi(e)\notin B\}|.
	\]
	Thus $d_{e(H)-1}(\chi)$ is the minimum number of edges whose colours must be changed so that the resulting colouring uses at most $e(H)-1$ colours.
	
	\begin{theorem}\label{thm:stability}
		Let $H$ be a fixed non-matching graph with $e(H)\ge3$, and let $\ell$ be a fixed integer satisfying $e(H)\le\ell<(e(H)-1)^{e(H)/(e(H)-2)}$.  For every $\eps>0$ there are constants $A=A(H,\ell,\eps)>0$ and $\gamma=\gamma(H,\ell,\eps)>0$ such that, if $p\ge A n^{-1/m_2(H)}$, then with high probability
		\[
		\bigl|\{\chi\in\cR_{H,\ell}(G(n,p)):
		d_{e(H)-1}(\chi)\ge\eps e(G(n,p))\}\bigr|
		\le(e(H)-1)^{(1-\gamma)e(G(n,p))}.
		\]
		Consequently, the number of such colourings is $o(R_{H,\ell}(G(n,p)))$.
	\end{theorem}
	
	\subsection{The rainbow entropy--Tur\'an density}
	
	We next give explicit graph-theoretic information about $\lambda(H,\ell)$.  For $1\le j\le e(H)-1$, define
	\[
	\chi_j(H):=\min_{\substack{S\subseteq E(H)\\|S|=j}}\chi(H-S)
	\]
	and
	\[
	\pi_j(H):=
	\begin{cases}
		1-\dfrac1{\chi_j(H)-1},&\chi_j(H)\ge3,\\[1.2ex]
		0,&\chi_j(H)\le2.
	\end{cases}
	\]
	For a finite graph family $\mathcal F$, let $\operatorname{ex}(n,\mathcal F)$ denote the maximum number of edges in an $n$-vertex graph containing no member of $\mathcal F$, and call
	$\lim_{n\to\infty}\operatorname{ex}(n,\mathcal F)/\binom n2$ its \emph{Tur\'an density}.  Thus $\pi_j(H)$ is the Tur\'an density of the finite family
	\[
	\mathcal H_j^-:=\{H-S:S\subseteq E(H),\ |S|=j\}.
	\]
	We write $\mathcal H^-=\mathcal H_1^-$ and $\pi_-(H)=\pi_1(H)$, and refer to the sequence $(\pi_j(H))_{1\le j\le e(H)-1}$ as the \emph{edge-deletion density profile} of $H$.  Templates constructed from $\mathcal H_j^-$-free graphs and an Erd\H{o}s--Stone--Simonovits upper bound give the following lower and upper estimates, respectively.
	
	\begin{theorem}\label{thm:deletion-bounds}
		For every fixed integer $\ell\ge e(H)$,
		\begin{equation}\label{eq:lambda-lower-profile}
			\lambda(H,\ell)\ge
			\max_{1\le j\le e(H)-1}
			\bigl(\pi_j(H)\log \ell+(1-\pi_j(H))\log j\bigr),
		\end{equation}
		and
		\begin{equation}\label{eq:lambda-upper-minus}
			\lambda(H,\ell)
			\le
			\pi_-(H)\log \ell+\bigl(1-\pi_-(H)\bigr)\log(e(H)-1).
		\end{equation}
		Consequently,
		\[
		\lim_{\ell\to\infty}\frac{\lambda(H,\ell)}{\log \ell}=\pi_-(H).
		\]
		Moreover, $\lambda(H,\ell)=\log(e(H)-1)$ for every integer $\ell\ge e(H)$ if and only if $H-e$ is bipartite for some $e\in E(H)$.
	\end{theorem}
	
	Together, the results determine $\lambda(H,\ell)$ throughout the Hall range, determine it for every $\ell\ge e(H)$ when $\pi_-(H)=0$ (equivalently, some $H-e$ is bipartite), and determine its leading term as $\ell\to\infty$.  Its value for the remaining pairs $(H,\ell)$ is a deterministic complete-graph template problem, discussed further in \cref{subsec:further-lambda}.
	
	We close the introduction with the main ideas of the proof.  In the below-threshold regime, choose a subgraph $F\subseteq H$ attaining $m_2(H)$.  If $p$ is at most a sufficiently small constant multiple of $n^{-1/m_2(H)}$, then, for every prescribed positive proportion, the constant can be chosen so that with high probability no more than that proportion of the edges of $G(n,p)$ lie in a copy of $F$.  Restrict those edges to $e(F)-1$ colours and colour all remaining edges freely with all $\ell$ colours.
	
	For the above-threshold regime, apply multicolour sparse regularity and the K{\L}R embedding theorem \cite{CGSS,GerkeSteger}.  Each cluster pair that is regular in every colour is assigned the palette of colours whose relative density exceeds a fixed positive threshold.  If the palettes on a reduced copy of $H$ had a system of distinct representatives, sparse embedding would produce a rainbow copy of $H$.  The random count is therefore reduced to the entropy--Tur\'an problem for $H$-SDR-free partial templates.  The exact local inequality
	\[
	\prod_{i=1}^{e(H)}|A_i|
	\le\max_{2\le s\le e(H)}(s-1)^s\ell^{e(H)-s}
	\]
	identifies the Hall range.  Below its upper endpoint, the strict gap from equality, combined with a mixing lemma, shows that a reduced template whose entropy differs from the upper bound by $o(m^2)$ assigns one fixed $(e(H)-1)$-palette to all but $o(m^2)$ pairs; this yields counting stability.  Finally, the edge-deletion density profile connects the remaining complete-graph problem to ordinary Tur\'an densities.
	
	\section{Preliminaries and local palette inequalities}
	
	All logarithms are natural unless a base is displayed.  For a graph $G$, write $v(G)$ and $e(G)$ for its numbers of vertices and edges.  For $S\subseteq V(G)$, let $e_G(S):=e(G[S])$, and for disjoint $U,W\subseteq V(G)$ let $e_G(U,W)$ be the number of edges with one endpoint in $U$ and the other in $W$.  We use $(x)_t=x(x-1)\cdots(x-t+1)$.  Whenever $\phi:V(H)\to[m]$ is an injection, the same symbol denotes the induced injection on edges, defined by
	\[
	\phi(uv):=\{\phi(u),\phi(v)\}\in E(K_m)
	\qquad (uv\in E(H)).
	\]
	
	\subsection{Sparse regularity and embedding}
	
	For disjoint vertex sets $X,Y$ in a graph $J$, define the relative $p$-density
	\[
	d_{J,p}(X,Y)=\frac{e_J(X,Y)}{p|X||Y|}.
	\]
	The pair $(X,Y)$ is $(\eps,p)$-\emph{regular} in $J$ if
	$|d_{J,p}(X',Y')-d_{J,p}(X,Y)|\le\eps$
	whenever $|X'|\ge\eps|X|$ and $|Y'|\ge\eps|Y|$.  A graph $J$ on $n$ vertices is \emph{$(\eta,p,D)$-upper-uniform} if
	$e_J(U,W)\le Dp|U||W|$
	for all disjoint sets $U,W\subseteq V(J)$ with $|U|,|W|\ge\eta n$.
	
	The above-threshold transference argument first requires one partition that is regular simultaneously for all colour classes.  We use the following form of the multicolour sparse regularity lemma; see, for example, \cite{Kohayakawa,GerkeSteger}.
	
	\begin{lemma}[Kohayakawa~\cite{Kohayakawa}; Gerke--Steger~\cite{GerkeSteger}]\label{lem:regularity}
		For every $\eps>0$, $D\ge1$, $m_0\in\N$ and $\ell\in\N$ there are $\eta>0$ and $M\ge m_0$ such that the following holds.  If $J_1,\ldots,J_\ell$ are spanning subgraphs of an $(\eta,p,D)$-upper-uniform graph $J$ on $n$ vertices, then there is a partition
		\[
		V(J)=V_0\cup V_1\cup\cdots\cup V_k
		\]
		with $m_0\le k\le M$, $|V_0|\le\eps n$, and $|V_1|=\cdots=|V_k|$, such that all but at most $\eps\binom{k}{2}$ pairs $(V_i,V_j)$ are $(\eps,p)$-regular in every $J_a$.
	\end{lemma}
	
	The next input turns an SDR on a copy of $H$ in the reduced template into an actual rainbow copy in the random host.  Given a graph $G$ and pairwise disjoint sets $(U_x)_{x\in V(H)}$ in $V(G)$, a \emph{transversal copy} of $H$ is a copy with vertices $u_x\in U_x$, $x\in V(H)$, and edges $u_xu_y$ for every $xy\in E(H)$.  The copy need not be induced.  We use the different-density form of the K{\L}R theorem proved in \cite[Proposition~4.2]{CGSS}.
	
	\begin{lemma}[Conlon--Gowers--Samotij--Schacht~\cite{CGSS}]\label{lem:embedding}
		Let $H$ be fixed and let $\alpha>0$.  There is $\eps_0=\eps_0(H,\alpha)>0$ such that, for every fixed $M\in\N$ and every $0<\eps\le\eps_0$, there is $A=A(H,\alpha,M,\eps)>0$ with the following property.  If
		$p\ge A n^{-1/m_2(H)},$
		then with high probability $G=G(n,p)$ satisfies the following.
		
		Let $(U_x)_{x\in V(H)}$ be disjoint sets of a common size $u\ge n/(2M)$.  For every $xy\in E(H)$ let $J_{xy}\subseteq G[U_x,U_y]$ be a bipartite subgraph which is $(\eps,p)$-regular and has at least $\alpha pu^2$ edges.  Then the union of the graphs $J_{xy}$ contains a transversal copy of $H$.
	\end{lemma}
	
	Indeed, \cite[Proposition~4.2]{CGSS} allows the regular pairs corresponding to different edges of $H$ to have different densities.  Once the regularity error is chosen below $\eps_0(H,\alpha)$, the constant in the lower bound for $p$ may absorb the lower bound $n/(2M)$ on the common cluster size.
	
	\subsection{Uniform properties of the random host}
	
	The regularity reduction also needs estimates that hold simultaneously for every cluster and every vertex set of size at most $\rho n$.  The following lemma collects the required consequences of Chernoff's inequality and union bounds.
	
	\begin{lemma}\label{lem:host}
		Fix $M\in\N$ and $\xi,\rho,\eta>0$.  There is
		$A=A(M,\xi,\rho,\eta)>0$ such that, if $p\ge A/n$, then with high probability $G=G(n,p)$ satisfies all of the following.
		\begin{enumerate}
			\item $e(G)=(1\pm\xi)p\binom n2$.
			\item For all disjoint $U,W\subseteq V(G)$ with $|U|,|W|\ge n/(2M)$,
			$
			e_G(U,W)=(1\pm\xi)p|U||W|.
			$
			\item For every $U\subseteq V(G)$ with $|U|\ge n/(2M)$,
			$
			e_G(U)\le(1+\xi)p\binom{|U|}{2}.
			$
			\item For every $S\subseteq V(G)$ with $|S|\le\rho n$, the number of edges incident with $S$ satisfies
			$
			e_G(S)+e_G(S,V(G)\setminus S)\le4\rho pn^2.
			$
			\item The graph $G$ is $(\eta,p,2)$-upper-uniform.
		\end{enumerate}
	\end{lemma}
	
	\begin{proof}
		Parts (1)--(3) follow from Chernoff's inequality and a union bound.  For (2) and (3), the binomial variables $e_G(U,W)$ and $e_G(U)$ have expectation $\Omega_{M}(pn^2)=\Omega_M(An)$, whereas there are at most $4^n$ ordered pairs $(U,W)$ and at most $2^n$ sets $U$.
		
		For (4), fix $S$ with $|S|\le\rho n$.  The number $e_G(S)+e_G(S,V(G)\setminus S)$ is stochastically dominated by a binomial random variable of mean at most $\rho pn^2$.  A Chernoff bound and a union bound over at most $2^n$ choices of $S$ prove the assertion when $A$ is sufficiently large.  Finally, (5) follows in the same way: for disjoint $U,W$ with $|U|,|W|\ge\eta n$, the probability that $e_G(U,W)>2p|U||W|$ is $\exp(-\Omega_\eta(pn^2))$, which can be summed over all choices of $U,W$.
	\end{proof}
	
	\subsection{Local Hall-palette inequalities}
	
	A \emph{palette} is a nonempty subset of the available colour set.  A family of palettes $A_1,\ldots,A_q$ has a \emph{system of distinct representatives} (SDR) if there are pairwise distinct colours $a_i\in A_i$.
	
	The following extremal inequality is the local input behind the Hall range.  Hall's theorem supplies a subfamily of $s$ palettes whose union has size at most $s-1$; these palettes are therefore confined to at most $s-1$ colours, while each remaining palette has size at most $\ell$.
	
	\begin{lemma}\label{lem:hallgeneral}
		Let $\ell\ge q\ge3$, and let $A_1,\ldots,A_q\subseteq[\ell]$ be nonempty sets with no SDR\@.  Then
		\[
		\prod_{i=1}^q|A_i|
		\le M_q(\ell):=\max_{2\le s\le q}(s-1)^s \ell^{q-s}.
		\]
		Equality holds if and only if, after permuting the indices, there exist a maximising value $s$ and a set $B\subseteq[\ell]$ of size $s-1$ such that
		\[
		A_1=\cdots=A_s=B,
		\qquad
		A_{s+1}=\cdots=A_q=[\ell].
		\]
	\end{lemma}
	
	\begin{proof}
		By Hall's theorem \cite{Hall}, there is $I\subseteq[q]$ with $s:=|I|\ge2$ and
		$
		\left|\bigcup_{i\in I}A_i\right|\le s-1.
		$
		Thus $|A_i|\le s-1$ for $i\in I$, while trivially $|A_i|\le \ell$ for $i\notin I$.  Hence
		\[
		\prod_{i=1}^q|A_i|\le(s-1)^s \ell^{q-s}\le M_q(\ell).
		\]
		Equality in the first inequality forces every $A_i$, $i\in I$, to have size $s-1$ and to lie in a common $(s-1)$-set, so they are all equal to that set; it also forces $A_i=[\ell]$ for $i\notin I$.  Equality in the final inequality requires $s$ to maximise the displayed expression.  Conversely, every family of the stated form has no SDR and attains equality.
	\end{proof}
	
	The next proposition evaluates the one-variable maximum in \cref{lem:hallgeneral}.  It identifies exactly the Hall range on which a common $(q-1)$-palette maximises the product of the $q$ palette sizes among families with no SDR, and it identifies the only equality form below the upper endpoint.
	
	\begin{proposition}\label{prop:localthreshold}
		For integers $\ell\ge q\ge3$,
		\[
		M_q(\ell)\le(q-1)^q
		\quad\Longleftrightarrow\quad
		\ell\le(q-1)^{q/(q-2)}.
		\]
		If $\ell<(q-1)^{q/(q-2)}$, then equality in
		$
		\prod_{i=1}^q|A_i|\le(q-1)^q
		$
		for a no-SDR family is possible only when
		$
		A_1=\cdots=A_q=B
		$
		for some $B\in\binom{[\ell]}{q-1}$.
	\end{proposition}
	
	\begin{proof}
		First note that
		\[
		M_q(\ell)\ge (2-1)^2\ell^{q-2}=\ell^{q-2}.
		\]
		So $M_q(\ell)\le(q-1)^q$ implies
		$\ell\le(q-1)^{q/(q-2)}$.
		
		Conversely, suppose $\ell\le(q-1)^{q/(q-2)}$.  The term corresponding to $s=2$ is at most $(q-1)^q$ by this assumption.  For $3\le s\le q$, it is enough to prove
		\[
		(s-1)^s
		\le(q-1)^{q(s-2)/(q-2)}.
		\]
		Equivalently,
		\[
		\frac{s\log(s-1)}{s-2}
		\le
		\frac{q\log(q-1)}{q-2}.
		\]
		The function
		\[
		f(x)=\frac{x\log(x-1)}{x-2},\qquad x\ge3,
		\]
		is increasing.  Indeed, with $y=x-1\ge2$, the numerator of $f'(x)$ is
		$
		g(y):=y-y^{-1}-2\log y.
		$
		Since $g(1)=0$ and
		$
		g'(y)={(y-1)^2}/{y^2}\ge0,
		$
		we have $g(y)\ge0$ for $y\ge1$.  This proves the inequality.  The equality statement follows from \cref{lem:hallgeneral}; below the endpoint, the only maximising value is $s=q$.
	\end{proof}
	
	For fixed $q,\ell$ with $\ell<(q-1)^{q/(q-2)}$, there are only finitely many $q$-tuples of palettes.  Hence there is $\Delta=\Delta(q,\ell)>0$ such that every no-SDR $q$-tuple which is not of the form
	$(B,\ldots,B)$ with $|B|=q-1$ satisfies
	\begin{equation}\label{eq:localgap}
		q\log(q-1)-\sum_{i=1}^q\log|A_i|\ge\Delta.
	\end{equation}
	
	\subsection{A mixing lemma}
	
	For stability, the palette identified on each copy of $H$ must be shown to agree across most pairs of the reduced graph.  The following mixing lemma supplies this step: if few pairs of adjacent edges have different labels, then one label occurs on almost all edges of the complete reduced graph.
	
	\begin{lemma}\label{lem:mixing}
		Let $m\ge2$, and let $\tau:E(K_m)\to[r]$ be an edge-labelling.  Let $B(\tau)$ be the number of unordered pairs of adjacent edges receiving different labels.  Then some label $a\in[r]$ is used on all but at most
		$
		\frac{3r B(\tau)}{m-1}
		$
		edges of $K_m$.
	\end{lemma}
	
	\begin{proof}
		For $v\in V(K_m)$ let $d_i(v)$ be the number of incident edges with label $i$, let
		$D_v=\max_i d_i(v)$, and set $r_v=m-1-D_v$.  The number of differently labelled pairs of edges meeting at $v$ is
		\[
		\sum_{i<j}d_i(v)d_j(v)\ge D_v r_v\ge\frac{m-1}{r}r_v.
		\]
		Since each adjacent pair has a unique common endpoint,
		\[
		\sum_v r_v\le\frac{rB(\tau)}{m-1}.
		\]
		For each vertex $v$, choose a label $b(v)$ attaining $D_v$; ties are broken arbitrarily.  If $b(u)\ne b(v)$, then
		$\tau(uv)$ cannot be equal to both $b(u)$ and $b(v)$.  Hence
		$
		\tau(uv)\ne b(u)$ or $\tau(uv)\ne b(v),
		$
		so the edge $uv$ is counted by $r_u$ or $r_v$ (possibly by both).
		Therefore the number of edges joining two distinct classes defined by $b$
		is at most $\sum_v r_v$.  If these classes have sizes
		$m_1,\ldots,m_r$, then
		$
		\sum_{i<j}m_i m_j\le\sum_v r_v.
		$
		Let $M=\max_i m_i$.  Since
		\[
		\sum_{i<j}m_i m_j
		=\frac12\sum_i m_i(m-m_i)
		\ge\frac{m(m-M)}2,
		\]
		some class, say the class with label $a$, has size
		$
		M\ge m-{2\sum_v r_v}/{m}.
		$
		Every edge not labelled $a$ either meets the complement of this class
		or has both endpoints in the class and is therefore counted among the
		incidences whose label differs from $a$ there.  The first type contributes at most
		$m\left(m-M\right)\le 2\sum_v r_v$
		edges, while the second type contributes at most $\sum_v r_v$ edges.
		Hence the number of edges not labelled $a$ is at most
		\[
		3\sum_v r_v
		\le \frac{3rB(\tau)}{m-1},
		\]
		as required.
	\end{proof}
	
	\section{The entropy--Tur\'an density and partial templates}\label{sec:deterministic-entropy}
	
	We now formalise the complete-graph entropy--Tur\'an problem used in
	\cref{thm:lambda-transfer}.  An $\ell$-palette template on $K_n$ is a map
	\[
	\cP:E(K_n)\longrightarrow 2^{[\ell]}\setminus\{\emptyset\}.
	\]
	It is \emph{$H$-SDR-free} if, on every copy of $H$ in $K_n$, the palettes
	assigned to its edges have no system of distinct representatives.  Following
	the template-entropy terminology of \cite{FROU}, define the logarithmic
	entropy--Tur\'an function by
	\begin{equation}\label{eq:exlog-definition}
		\operatorname{ex}_{\log}(n;H,\ell)
		:=\max\left\{
		\sum_{e\in E(K_n)}\log|\cP(e)|:
		\cP\text{ is an $H$-SDR-free $\ell$-palette template on $K_n$}
		\right\}.
	\end{equation}
	For a fixed template, the exponential of the sum in
	\eqref{eq:exlog-definition} is exactly the number of colourings realising that
	template.  Equivalently, the sum is the Shannon entropy of the product measure
	which chooses the colour of each edge independently and uniformly from its
	palette.
	
	Define the rainbow entropy--Tur\'an density by
	\begin{equation}\label{eq:lambda-definition}
		\lambda(H,\ell):=
		\inf_{n\ge v(H)}
		\frac{\operatorname{ex}_{\log}(n;H,\ell)}{\binom n2}.
	\end{equation}
	Thus $\lambda(H,\ell)$ is the limit of the maximum average logarithmic palette
	size among templates whose palettes exclude every rainbow copy of $H$.
	
	The reduced templates arising from sparse regularity are not defined on every
	pair.  To incorporate this feature, let $m\ge v(H)$ and
	$X\subseteq E(K_m)$.  A \emph{partial $\ell$-palette template with omitted-pair
		set $X$} is a map
	\[
	\cP_X:E(K_m)\setminus X
	\longrightarrow 2^{[\ell]}\setminus\{\emptyset\}.
	\]
	It is \emph{$H$-SDR-free} if, for every injection
	$\phi:V(H)\to[m]$ satisfying $\phi(E(H))\cap X=\emptyset$, the palettes
	\[
	\{\cP_X(\phi(e)):e\in E(H)\}
	\]
	have no SDR.  The subscript records the pairs on which the template is not defined, and we write
	\begin{equation}\label{eq:partial-entropy-definition}
		\Ent_X(\cP_X)
		:=\sum_{e\in E(K_m)\setminus X}\log|\cP_X(e)|.
	\end{equation}
	The next proposition records the two averaging facts needed later: the
	normalised complete-template maximum converges, and allowing the template to
	be undefined on at most $\eta m^2$ pairs increases the normalised upper bound
	by at most a prescribed error when $\eta$ is sufficiently small.
	
	\begin{proposition}\label{prop:lambda-partial}
		Fix $\ell\ge e(H)$.
		\begin{enumerate}
			\item The sequence
			$\operatorname{ex}_{\log}(n;H,\ell)/\binom n2$ is nonincreasing for
			$n\ge v(H)$.  In particular,
			\[
			\lambda(H,\ell)=
			\lim_{n\to\infty}
			\frac{\operatorname{ex}_{\log}(n;H,\ell)}{\binom n2}.
			\]
			\item For every $\zeta>0$ there are $\eta>0$ and $m_0$ such that the following holds.  Let $m\ge m_0$, let $X\subseteq E(K_m)$ satisfy $|X|\le\eta m^2$, and let $\cP_X$ be an $H$-SDR-free partial $\ell$-palette template with omitted-pair set $X$.  Then
			\[
			\Ent_X(\cP_X)\le
			\bigl(\lambda(H,\ell)+\zeta\bigr)\binom m2.
			\]
		\end{enumerate}
	\end{proposition}
	
	\begin{proof}
		For (1), let $n\ge t\ge v(H)$ and let $\cP$ be an $H$-SDR-free template on $K_n$.  The restriction of $\cP$ to every $t$-set is again $H$-SDR-free.  Averaging its entropy over all $t$-subsets gives
		\[
		\frac{1}{\binom n2}
		\sum_{e\in E(K_n)}\log|\cP(e)|
		=\frac{1}{\binom t2}
		\mathbb E_T
		\sum_{e\in E(T)}\log|\cP(e)|
		\le \frac{\operatorname{ex}_{\log}(t;H,\ell)}{\binom t2}.
		\]
		Taking the maximum over $\cP$ proves the required monotonicity.
		
		For (2), choose $t\ge v(H)$ so that
		\[
		\frac{\operatorname{ex}_{\log}(t;H,\ell)}{\binom t2}
		\le\lambda(H,\ell)+\frac{\zeta}{2}.
		\]
		Let $T$ be a uniformly random $t$-subset of $[m]$ and let
		\[
		W(T):=\sum_{e\in E(T)\setminus X}\log|\cP_X(e)|.
		\]
		If $E(T)\cap X=\emptyset$, then the restriction of $\cP_X$ to $E(T)$ is an
			$H$-SDR-free $\ell$-palette template on the complete graph with vertex set
		$T$, and hence
		$W(T)\le\operatorname{ex}_{\log}(t;H,\ell)$.  In all cases
		$W(T)\le\binom t2\log \ell$.  Moreover,
		\[
		\mathbb P(E(T)\cap X\ne\emptyset)
		\le |X|\frac{\binom t2}{\binom m2}
		\le 4\eta\binom t2
		\]
		for $m$ sufficiently large.  Since
		\[
		\mathbb E W(T)=
		\frac{\binom t2}{\binom m2}\Ent_X(\cP_X),
		\]
		we obtain
		\[
		\frac{\Ent_X(\cP_X)}{\binom m2}
		\le
		\frac{\operatorname{ex}_{\log}(t;H,\ell)}{\binom t2}
		+4\eta\binom t2\log \ell.
		\]
		Choosing $\eta$ so that the last error is at most $\zeta/2$ proves the assertion.
	\end{proof}
	
	\subsection{The Hall range and stability}
	
	We now specialise the partial-template framework to the Hall range.
	The next proposition first gives a global entropy upper bound.  Below the
	upper endpoint of the range, it also states the stability conclusion as a
	one-way implication: entropy within $\kappa m^2$ of the upper bound forces
	closeness to one common $(e(H)-1)$-palette.  Equivalently, failure of this
	closeness condition implies a quadratic entropy deficit; no mutual
	exclusivity of the two properties is asserted.  The proof double-counts
	embeddings of $H$, uses the local gap in \eqref{eq:localgap} to bound the
	number of embeddings not having the equality form, and then applies
	\cref{lem:mixing}.
	
	\begin{proposition}\label{prop:hall-stability}
		Let $m\ge v(H)$, let $X\subseteq E(K_m)$, and let $\ell\ge e(H)$ satisfy
		$\ell\le (e(H)-1)^{e(H)/(e(H)-2)}$.  Every $H$-SDR-free partial $\ell$-palette template $\cP_X$ with omitted-pair set $X$ satisfies
		\begin{equation}\label{eq:partial-hall-upper}
			\Ent_X(\cP_X)
			\le\binom m2\log(e(H)-1)
			+e(H)|X|\log\frac{\ell}{e(H)-1}.
		\end{equation}
		
		If $\ell<(e(H)-1)^{e(H)/(e(H)-2)}$, then for every $\theta>0$ there are
		$\eta,\kappa>0$ and $m_0\ge v(H)$ such that the following holds whenever $m\ge m_0$ and $|X|\le\eta m^2$.  If
		\[
			\Ent_X(\cP_X)
			>\binom m2\log(e(H)-1)-\kappa m^2,
		\]
		then there is a set $B\in\binom{[\ell]}{e(H)-1}$ for which
		\begin{equation}\label{eq:paletteclose}
			|X|+|\{e\notin X:\cP_X(e)\ne B\}|\le\theta m^2.
		\end{equation}
		Equivalently, if no such set $B$ exists, then
		\begin{equation}\label{eq:entropydeficit}
			\Ent_X(\cP_X)
			\le\binom m2\log(e(H)-1)-\kappa m^2.
		\end{equation}
	\end{proposition}
	
	\begin{proof}
		Let $\Phi$ be the set of injections $\phi:V(H)\to[m]$.  For a fixed edge $xy\in E(K_m)$, let $L$ be the number of pairs $(\phi,f)$ with $\phi\in\Phi$, $f\in E(H)$ and $\phi(f)=xy$.  Then
		$
		L=2e(H)(m-2)_{v(H)-2}.
		$
		Since an injection maps at most one edge of $H$ to $xy$, exactly $L$ injections in $\Phi$ satisfy $xy\in\phi(E(H))$.  Counting all pairs $(\phi,f)$ first by the injection $\phi$ and then by the image edge $\phi(f)$ gives
		\begin{equation}\label{eq:incidenceidentity}
			L\binom m2=e(H)(m)_{v(H)}.
		\end{equation}
		For each edge $e\in X$, exactly $L$ injections $\phi$ satisfy
		$e\in\phi(E(H))$.  Hence, by the union bound, the number of injections
		$\phi$ for which
		$\phi(E(H))\cap X\ne\varnothing$
		is at most $L|X|$.

		We now make the weighted double count explicit.  For $e\in E(K_m)\setminus X$, let
		$w(e):=\log|\cP_X(e)|$, and consider the weighted incidence sum
		\[
			S:=\sum_{\phi\in\Phi}
			\sum_{\substack{f\in E(H)\\ \phi(f)\notin X}}
			w\bigl(\phi(f)\bigr).
		\]
		Count $S$ first by the image edge $e=\phi(f)$.  Every edge
		$e\in E(K_m)\setminus X$ occurs in exactly $L$ pairs $(\phi,f)$, and hence
		\[
			S=L\sum_{e\in E(K_m)\setminus X}w(e)
			=L\Ent_X(\cP_X).
		\]
		Count $S$ instead by the injection $\phi$.  If $\phi(E(H))\cap X=\varnothing$, then the local palette bounds in \cref{lem:hallgeneral,prop:localthreshold} give a contribution of at most $e(H)\log(e(H)-1)$.  At most $L|X|$ injections meet $X$, and the contribution of each such injection from image edges outside $X$ is at most $e(H)\log\ell$.  Relative to the bound $e(H)\log(e(H)-1)$, each of these injections therefore incurs an additional cost of at most $e(H)\log(\ell/(e(H)-1))$.  Consequently,
		\[
		S=L\Ent_X(\cP_X)
		\le e(H)(m)_{v(H)}\log(e(H)-1)
		+L|X|e(H)\log\frac{\ell}{e(H)-1}.
		\]
		Using \eqref{eq:incidenceidentity} proves \eqref{eq:partial-hall-upper}.
		
		Now assume $\ell<(e(H)-1)^{e(H)/(e(H)-2)}$, and let $\Delta=\Delta(e(H),\ell)>0$ be the gap supplied by \eqref{eq:localgap}.  Call an injection $\phi$ \emph{exceptional} if $\phi(E(H))\cap X\ne\emptyset$, or if $\phi(E(H))\cap X=\emptyset$ but the palettes on its $H$-edges are not all equal to a common $(e(H)-1)$-set.  Thus ``exceptional'' has only this local meaning in the present proof.  Let $D$ be the number of exceptional injections and let $D_X$ be the number satisfying $\phi(E(H))\cap X\ne\emptyset$.  The incidence count above gives
		\begin{equation}\label{eq:DXbound}
			D_X\le L|X|\le L\eta m^2.
		\end{equation}
		For every exceptional injection avoiding $X$, the local gap \eqref{eq:localgap} lowers the total contribution of its $e(H)$ image edges by at least $\Delta$ relative to $e(H)\log(e(H)-1)$.  An injection whose image contains an edge of $X$ contributes at most $e(H)\log\ell$ on its image edges outside $X$.  Consequently,
		\[
		L\Ent_X(\cP_X)
		\le e(H)(m)_{v(H)}\log(e(H)-1)-\Delta(D-D_X)
		+e(H)D_X\log\frac{\ell}{e(H)-1}.
		\]
		If
		\[
		\Ent_X(\cP_X)
		>\binom m2\log(e(H)-1)-\kappa m^2,
		\]
		then \eqref{eq:incidenceidentity} and \eqref{eq:DXbound} imply
		\[
		\Delta D
		<L\kappa m^2+
		\left(\Delta+e(H)\log\frac{\ell}{e(H)-1}\right)D_X.
		\]
		Since $L=\Theta_H(m^{v(H)-2})$, the number of exceptional injections satisfies
		\begin{equation}\label{eq:exceptionalcount}
			D\le C_1(H,\ell)(\kappa+\eta)m^{v(H)}.
		\end{equation}
		
		Let
		\[
		Y=\{e\notin X:\cP_X(e)\notin\tbinom{[\ell]}{e(H)-1}\}.
		\]
		Every injection containing an edge of $Y$ is exceptional.  Each edge of $Y$ lies in exactly $L$ edge--embedding incidences, whereas an exceptional injection contains at most $e(H)$ such edges.  Hence
		$
		L|Y|\le e(H)D,
		$
		and therefore
		\begin{equation}\label{eq:Ysmall}
			|Y|\le C_2(H,\ell)(\kappa+\eta)m^2.
		\end{equation}
		For $e\notin X\cup Y$, label $e$ by the set $\cP_X(e)\in\binom{[\ell]}{e(H)-1}$.
		
		Fix two adjacent edges of $H$.  Every pair of adjacent edges of $K_m$ outside $X\cup Y$ that receive different labels has at least
		$2(m-3)_{v(H)-3}$ extensions sending these two fixed edges of $H$ to that pair.  Each extension is exceptional, since an injection that is not exceptional has one common palette on all of its $H$-edges.  Thus \eqref{eq:exceptionalcount} shows that the number of differently labelled adjacent pairs outside $X\cup Y$ is at most
		$C_3(H,\ell)(\kappa+\eta)m^3$.  Extend the labelling arbitrarily over $X\cup Y$.  Adjacent pairs meeting $X\cup Y$ contribute at most
		$2m(|X|+|Y|)$ additional discrepancies.  Applying \cref{lem:mixing} with
		$r=\binom{\ell}{e(H)-1}$, we obtain a set $B\in\binom{[\ell]}{e(H)-1}$ which labels all but
		$C_4(H,\ell)(\kappa+\eta)m^2$ edges.  Together with \eqref{eq:Ysmall} and $|X|\le\eta m^2$, this gives \eqref{eq:paletteclose} after choosing first $\eta$ and then $\kappa$ sufficiently small in terms of $\theta$.
	\end{proof}
	
	Taking $X=\emptyset$ in the upper bound \eqref{eq:partial-hall-upper} determines the entropy--Tur\'an density throughout the Hall range.
	
	\begin{corollary}\label{cor:hall-lambda}
		If
		\[
		e(H)\le\ell\le(e(H)-1)^{e(H)/(e(H)-2)},
		\]
		then
		$
		\lambda(H,\ell)=\log(e(H)-1).
		$
	\end{corollary}
	
	\begin{proof}
		The constant template $\cP(e)=B$ for every $e\in E(K_m)$, where
		$B\in\binom{[\ell]}{e(H)-1}$ is fixed, is $H$-SDR-free.  Hence
		$\operatorname{ex}_{\log}(m;H,\ell)\ge\binom m2\log(e(H)-1)$ for every $m$.
		For the reverse inequality, apply \eqref{eq:partial-hall-upper} with
		$X=\emptyset$ to every $H$-SDR-free template on $K_m$.  Taking the maximum
		over these templates, dividing by $\binom m2$, and then letting $m\to\infty$
		completes the proof.
	\end{proof}

	\section{Proof of the random-graph results}
	
	\subsection{The below-threshold regime}
	
	Choose a subgraph $F\subseteq H$ attaining $m_2(H)$.  We may take $F$ without isolated vertices; in particular, $e(F)\ge2$.  Since $F$ attains $m_2(H)$, it is $2$-balanced and
	\[
	m_2(F)=m_2(H)=\frac{e(F)-1}{v(F)-2}.
	\]
	For a graph $G$, let $X_F(G)$ be the number of labelled copies of $F$ in $G$ (that is, injective embeddings of $F$ into $G$), and let
	\[
	B_F(G)=\{e\in E(G):e\text{ belongs to a copy of }F\}.
	\]
	Each copy of $F$ has $e(F)$ edges, so $|B_F(G)|\le e(F)X_F(G)$.
	
	The below-threshold construction requires the edges lying in copies of $F$ to form at most a prescribed positive fraction of $E(G(n,p))$.  The next lemma proves this by a first-moment estimate together with variance control at the $m_2(F)$ scale.
	
	\begin{lemma}\label{lem:fewcore}
		For every $\rho>0$ there is $a=a(F,\rho)>0$ such that, if
		$p\le a n^{-1/m_2(F)}$, then with high probability
		\[
		|B_F(G(n,p))|\le\rho e(G(n,p)).
		\]
	\end{lemma}
	
	\begin{proof}
		It suffices to verify the assertion along every subsequence.  From any
		subsequence, one can pass to a further subsequence on which either
		$n^2p\to\infty$ or $n^2p=O(1)$.  First suppose $n^2p\to\infty$.  Since
		\[
		\mathbb E X_F=O_F\bigl(n^{v(F)}p^{e(F)}\bigr)
		=O_F\bigl(n^2p\,n^{v(F)-2}p^{e(F)-1}\bigr)
		\le O_F\bigl(a^{e(F)-1}n^2p\bigr),
		\]
		it remains to obtain concentration on the scale $n^2p$.  Two labelled copies of $F$ have nonzero covariance only when they share an edge.  Taking $J$ to be their edge-intersection, with isolated vertices omitted, we obtain
		\[
		\operatorname{Var}X_F
		\le C_F\sum_{\substack{J\subseteq F\\ e(J)\ge1}}
		n^{2v(F)-v(J)}p^{2e(F)-e(J)}.
		\]
		Since $2e(F)-e(J)-2\ge0$ and $p\le a n^{-1/m_2(F)}$,
		\[
		\frac{n^{2v(F)-v(J)}p^{2e(F)-e(J)}}{(n^2p)^2}
		\le a^{2e(F)-e(J)-2}n^{-v(J)+e(J)/m_2(F)}.
		\]
		If $J$ consists of one edge, the exponent is $-2+1/m_2(F)\le-1$.  If $v(J)\ge3$, the $2$-balancedness of $F$ gives
		$e(J)-1\le m_2(F)(v(J)-2)$, and once again
		\[
		-v(J)+\frac{e(J)}{m_2(F)}\le-2+\frac1{m_2(F)}\le-1.
		\]
		Thus $\operatorname{Var}X_F=o((n^2p)^2)$.  Chebyshev's inequality gives
		$
		X_F\le C_F a^{e(F)-1}n^2p+o(n^2p)
		$
		with high probability.  Since $e(G(n,p))=(1+o(1))n^2p/2$, choosing $a$ sufficiently small proves the result.
		
		If $n^2p=O(1)$, then $p=O(n^{-2})$.  Since $m_2(F)\ge1$, we have $v(F)\le e(F)+1$, and therefore $v(F)-2e(F)\le1-e(F)\le-1$.  Consequently,
		\[
		\mathbb E X_F=O\bigl(n^{v(F)}p^{e(F)}\bigr)=O\bigl(n^{v(F)-2e(F)}\bigr)=o(1).
		\]
		With high probability there is no copy of $F$, so $B_F(G(n,p))=\emptyset$ and the assertion follows.
	\end{proof}
	
	\begin{proof}[Proof of the below-threshold part of \cref{thm:lambda-transfer}]
		Let $B=B_F(G)$.  Colour the edges in $E(G)\setminus B$ arbitrarily with the $\ell$ available colours, and colour the edges of $B$ using a fixed set of $e(F)-1$ colours.  Every copy of $H$ contains a copy of $F$, all of whose $e(F)$ edges lie in $B$ and use at most $e(F)-1$ colours.  Hence the resulting colouring contains no rainbow $H$.  Therefore
		\[
		R_{H,\ell}(G)\ge \ell^{e(G)-|B|}(e(F)-1)^{|B|}.
		\]
		By \cref{lem:fewcore}, for each prescribed $\rho>0$ the constant $a$ can be chosen so that $|B|\le\rho e(G)$ with high probability.  Choose $\rho$ to satisfy
		\[
		\rho\log_\ell\frac{\ell}{e(F)-1}\le\delta.
		\]
		Then
		$R_{H,\ell}(G)\ge \ell^{(1-\delta)e(G)}$.
		The upper bound $R_{H,\ell}(G)\le\ell^{e(G)}$ follows because $G$ has exactly $\ell^{e(G)}$ unrestricted $\ell$-edge-colourings.
	\end{proof}
	
	\subsection{Above-threshold entropy transference}\label{subsec:dense-transference}
	
	Fix an integer $\ell\ge e(H)$ and $\delta>0$.  For each colouring, we record one partition that is regular simultaneously in all colours, the set of cluster pairs that fail regularity in at least one colour, and the resulting reduced palettes.  These data are used to prove the above-threshold part of the entropy-transference theorem; the estimate in the Hall range then follows from \cref{cor:hall-lambda}.
	
	Choose $\zeta>0$ and $0<\xi<1/3$ sufficiently small in terms of $\delta$, and then choose
	$\alpha>0$ so that $\alpha<1/(3\ell)$ and $\omega_\ell(\alpha)$, defined in
	\eqref{eq:binary-error} below, is sufficiently small in terms of $\delta$.  Apply
	\cref{prop:lambda-partial} with $\zeta$, obtaining $\eta_0>0$ and an integer
	$m_*$.  Choose $m_0\ge m_*$ sufficiently large and then choose
	\[
	0<\rho\le
	\min\{\eta_0,\alpha,\xi,\delta,1/2,\eps_0(H,\alpha)\}
	\]
	sufficiently small in terms of the displayed parameters,
	where $\eps_0(H,\alpha)$ is supplied by \cref{lem:embedding}.  Apply
	\cref{lem:regularity} with regularity error $\rho$ and $D=2$, obtaining
	$\eta>0$ and $M\ge m_0$.  Finally choose $A$ large enough for
	\cref{lem:embedding} with this $M$, for \cref{lem:host} with
	$M,\xi,\rho,\eta$, and to make the error from enumerating the regularity data below as small
	as required.  Let $G=G(n,p)$ with $p\ge A n^{-1/m_2(H)}$, and work on the
	event on which the conclusions of \cref{lem:embedding,lem:host} both hold.
	
	Fix $\chi\in\cR_{H,\ell}(G)$, and let $G_a$ be its colour-$a$ subgraph.  Apply \cref{lem:regularity} simultaneously to $G_1,\ldots,G_\ell$.  We obtain
	\[
	V(G)=V_0\cup V_1\cup\cdots\cup V_k,
	\]
	where $m_0\le k\le M$, $|V_0|\le\rho n$, and
	$|V_1|=\cdots=|V_k|=:u$.  Let $X\subseteq E(K_k)$ be the set of cluster pairs which fail to be regular in at least one colour; then $|X|\le\rho k^2$.
	
	For $ij\notin X$, define
	\begin{equation}\label{eq:reduced-palette}
	\cP_X(ij)=\{a\in[\ell]:e_{G_a}(V_i,V_j)\ge\alpha pu^2\}.
	\end{equation}
	Since $\xi<1/3$ and $\alpha<1/(3\ell)$, every such palette is nonempty: otherwise
	\[
	e_G(V_i,V_j)=\sum_{a=1}^{\ell}e_{G_a}(V_i,V_j)
	<\ell\alpha pu^2<\frac{pu^2}{3},
	\]
	contradicting $e_G(V_i,V_j)\ge(1-\xi)pu^2$ from \cref{lem:host}.
	
	The reduced template inherits the forbidden-rainbow condition.  An SDR on a reduced copy of $H$ would select pairwise distinct colour classes, each having relative $p$-density at least $\alpha$ on its assigned cluster pair, and the embedding lemma would then produce a rainbow copy of $H$ in $G$.
	
	\begin{claim}\label{cl:sdr-free}
		The partial template $\cP_X$ is $H$-SDR-free.
	\end{claim}
	
	\begin{proof}
		Suppose that an injection $\phi:V(H)\to[k]$ satisfies $\phi(E(H))\cap X=\emptyset$ and that the palettes on its $H$-edges have an SDR\@.  Choose pairwise distinct colours $a_e\in\cP_X(\phi(e))$, $e\in E(H)$.  For $e=xy$, the colour-$a_e$ graph on $(V_{\phi(x)},V_{\phi(y)})$ is regular and has at least $\alpha pu^2$ edges.  By \cref{lem:embedding}, these coloured pairs contain a transversal copy of $H$ whose edges receive the pairwise distinct colours $a_e$, a contradiction.
	\end{proof}
	
	We next count colourings represented by a fixed tuple of \emph{regularity data}
	\[
	\mathcal D=(V_0,V_1,\ldots,V_k;X,\cP_X).
	\]
	Here the vertex sets form a partition with the sizes specified above.  The term ``regularity data'' refers only to a tuple of this displayed form.  We say that $\chi\in\cR_{H,\ell}(G)$ is \emph{represented by $\mathcal D$} if $ij\in X$ exactly when $(V_i,V_j)$ fails to be $(\rho,p)$-regular in at least one colour subgraph of $\chi$, and if \eqref{eq:reduced-palette} holds for every $ij\notin X$ with those colour subgraphs.  Let $N(\mathcal D)$ denote the number of such colourings.
	
	Let $E_{\mathrm{out}}(\mathcal D)$ consist of the edges of $G$ that are incident with $V_0$, lie inside one of $V_1,\ldots,V_k$, or lie between $V_i$ and $V_j$ for some $ij\in X$.  Thus $E_{\mathrm{out}}(\mathcal D)$ is exactly the set of edges outside the regular cluster pairs indexed by $E(K_k)\setminus X$.  By parts (2)--(4) of \cref{lem:host}, its size is at most
	\[
	4\rho pn^2+(1+\xi)pk\binom u2+(1+\xi)p|X|u^2
	\le \sigma pn^2,
	\]
	where
	\begin{equation}\label{eq:sigmadef}
		\sigma:=4\rho+\frac{1+\xi}{2m_0}+(1+\xi)\rho
		\longrightarrow0
		\quad\text{as}\quad
		\rho,m_0^{-1},\xi\longrightarrow0.
	\end{equation}
	These edges may be coloured arbitrarily, contributing at most
	$\exp(\sigma\log\ell\,pn^2)$ possibilities.
	
	Fix $ij\notin X$, and write
	$N_{ij}=e_G(V_i,V_j)$ and $s_{ij}=|\cP_X(ij)|$.  Every colour outside the palette appears on fewer than $\alpha pu^2$ edges, so the total number of edges receiving colours outside the palette is at most
	$\ell\alpha pu^2$.  Hence the number of possible restrictions to $G[V_i,V_j]$ among colourings represented by $\mathcal D$ is at most
	\[
	s_{ij}^{N_{ij}}
	\sum_{t\le \ell\alpha pu^2}\binom{N_{ij}}t\ell^t.
	\]
	Since $N_{ij}=(1\pm\xi)pu^2$ and $\xi<1/3$, we have
	\[
		\frac{\ell\alpha pu^2}{N_{ij}}
		\le\frac{3\ell\alpha}{2}<\frac12.
	\]
	For $0\le x\le1$, let
	$h_2(x)=-x\log x-(1-x)\log(1-x)$, with $0\log0=0$, and define
	\begin{equation}\label{eq:binary-error}
		\omega_\ell(\alpha)
		:=\frac43h_2\left(\frac{3\ell\alpha}{2}\right)
		+\ell\alpha\log\ell.
	\end{equation}
	The estimate $\sum_{t\le xN}\binom Nt\le\exp(Nh_2(x))$ for
	$0\le x\le1/2$ gives
	\[
		\sum_{t\le \ell\alpha pu^2}\binom{N_{ij}}t\ell^t
		\le \exp\bigl(\omega_\ell(\alpha)pu^2\bigr).
	\]
	Since $\omega_\ell(\alpha)\to0$ as $\alpha\to0$, the number of restrictions to $G[V_i,V_j]$ is at most
	\begin{equation}\label{eq:paircount}
		\exp\bigl(\omega_\ell(\alpha)pu^2\bigr)
		s_{ij}^{(1+\xi)pu^2}.
	\end{equation}
	Multiplying \eqref{eq:paircount} over all pairs $ij\notin X$ and using \eqref{eq:sigmadef}, we obtain
	\begin{equation}\label{eq:fixed-data-count}
		\log N(\mathcal D)
		\le (1+\xi)pu^2\Ent_X(\cP_X)+\nu_0pn^2,
	\end{equation}
	where, for every prescribed $\nu_0>0$, the preceding parameters
	$\alpha,\rho,m_0^{-1}$ and $\xi$ can be chosen so that this value of $\nu_0$ is valid.
	
	There are at most $(M+1)^n$ vertex partitions and at most
	$2^{O_\ell(M^2)}$ choices for $X$ and the reduced palettes.  Thus the total number of possible tuples $\mathcal D$ is at most
	$
	(M+1)^n2^{O_\ell(M^2)}.
	$
	Since $m_2(H)\ge1$, we have $pn^2\ge An$.  When $m_2(H)>1$, the logarithm of the displayed quantity is $o(pn^2)$.  When $m_2(H)=1$, its ratio to $pn^2$ can be made as small as required by increasing $A$.  Consequently, after enlarging $\nu_0$ to a number $\nu>0$, for every nonempty collection $\mathcal S$ of regularity data we have
	\begin{equation}\label{eq:data-count}
		\log\sum_{\mathcal D\in\mathcal S}N(\mathcal D)
		\le
		\max_{\mathcal D\in\mathcal S}
		\bigl((1+\xi)pu_{\mathcal D}^{2}
		\Ent_{X_{\mathcal D}}(\cP_{\mathcal D})\bigr)
		+\nu pn^2.
	\end{equation}
	Here $X_{\mathcal D}$, $\cP_{\mathcal D}$ and $u_{\mathcal D}$ denote,
	respectively, the omitted-pair set, the partial template and the common
	size of $V_1,\ldots,V_k$ in $\mathcal D$.  For every prescribed positive value of $\nu$, the regularity parameters can first be chosen to control $\nu_0$, and $A$ can then be chosen so that \eqref{eq:data-count} holds with that value of $\nu$.
	
	\medskip
	\noindent\emph{Completion of the entropy estimate.}
	
	By the choices above, part~(2) of
	\cref{prop:lambda-partial} applies to the partial template in every tuple $\mathcal D$, and hence
	\[
	\Ent_X(\cP_X)
	\le\bigl(\lambda(H,\ell)+\zeta\bigr)\binom k2.
	\]
	Since $u\le n/k$, applying \eqref{eq:data-count} to the collection of all regularity data gives
	\[
	\log R_{H,\ell}(G)
	\le(1+\xi)\frac{pn^2}{2}
	\bigl(\lambda(H,\ell)+\zeta\bigr)+\nu pn^2.
	\]
	Choosing $\xi,\zeta,\nu$ sufficiently small and using
	$e(G)=(1\pm\xi)p\binom n2$, we obtain
	\begin{equation}\label{eq:lambda-random-upper}
		R_{H,\ell}(G)
		\le\exp\bigl((\lambda(H,\ell)+\delta)e(G)\bigr).
	\end{equation}
	
	For the reverse inequality, let $\cP_n$ be an $H$-SDR-free template on $K_n$ attaining
	$\operatorname{ex}_{\log}(n;H,\ell)$, and let
	$w_e=\log|\cP_n(e)|$.  Every colouring of $G(n,p)$ obtained by assigning each present edge $e$ a colour from $\cP_n(e)$ is rainbow-$H$-free, so
	\[
	\log R_{H,\ell}(G(n,p))
	\ge Z:=\sum_{e\in E(K_n)}\mathbf 1_{\{e\in G(n,p)\}}w_e.
	\]
	The variables in $Z$ are independent and bounded by $\log \ell$, and
	\[
	\operatorname{Var}Z\le p\binom n2(\log\ell)^2=O(pn^2).
	\]
	The constant $(e(H)-1)$-palette shows that $\lambda(H,\ell)\ge\log(e(H)-1)>0$, and
	\[
	\mathbb EZ=p\operatorname{ex}_{\log}(n;H,\ell)
	\ge p\lambda(H,\ell)\binom n2
	\]
	by the definition of $\lambda(H,\ell)$.  Since $p\binom n2\to\infty$, Bernstein's inequality gives
	\[
	Z\ge p\lambda(H,\ell)\binom n2-o(pn^2)
	\]
	with high probability.  Together with the concentration of $e(G(n,p))$, this yields
	\begin{equation}\label{eq:lambda-random-lower}
		R_{H,\ell}(G(n,p))
		\ge\exp\bigl((\lambda(H,\ell)-\delta)e(G(n,p))\bigr).
	\end{equation}
	Equations \eqref{eq:lambda-random-upper} and \eqref{eq:lambda-random-lower}, together with the below-threshold argument above, prove \cref{thm:lambda-transfer}.  Combining this theorem with \cref{cor:hall-lambda}, replacing the additive above-threshold error by $\delta\log(e(H)-1)$, and using the constant $(e(H)-1)$-palette for the above-threshold lower bound proves \cref{cor:hall-transition}.
	
	\subsection{Counting stability}
	
	Assume now that $e(H)\le\ell<(e(H)-1)^{e(H)/(e(H)-2)}$ and fix $\eps>0$.
	Choose $\theta>0$ sufficiently small in terms of $\eps$ and $\ell$, and apply the stability part of \cref{prop:hall-stability}, obtaining $\eta_0,\kappa>0$ and an integer $m_*$.  Repeat the regularity reduction from \cref{subsec:dense-transference}.  Choose, in order, $0<\xi<1/4$, $0<\alpha<1/(3\ell)$, $m_0\ge m_*$, $\rho>0$, and then $A>0$, subject to
	\[
	0<\rho<\min\{\eta_0,1/2,\eps_0(H,\alpha)\}
	\]
	and so that the quantity $\sigma$ in \eqref{eq:sigmadef} and the enumeration error $\nu$ in \eqref{eq:data-count} satisfy
	\begin{align}
		&\sigma+(1+\xi)\theta+\frac{\ell\alpha}{2}
		<\frac{\eps(1-2\xi)}2,\label{eq:stability-distance-errors}\\
		&(1+\xi)\kappa(1-\rho)^2-\nu-
		2\xi\log(e(H)-1)>0.\label{eq:stability-entropy-errors}
	\end{align}
	These choices are possible because $\theta$ is chosen first, $\omega_\ell(\alpha)\to0$ as $\alpha\to0$, $\sigma\to0$ as $\rho,m_0^{-1}\to0$ for fixed $\xi$, and the enumeration contribution to $\nu$ decreases when $A$ increases.  In particular,
	$|X|\le\rho k^2<\eta_0k^2$ for every tuple $\mathcal D$.
	
	Suppose that the partial template in $\mathcal D$ satisfies \eqref{eq:paletteclose} for a set
	$B\in\binom{[\ell]}{e(H)-1}$.  The edges whose colours lie outside $B$ are contained in the union of the following three classes:
	\begin{enumerate}
		\item the edges in $E_{\mathrm{out}}(\mathcal D)$, of which there are at most $\sigma pn^2$;
		\item the edges between pairs $V_i,V_j$ with $ij\notin X$ and $\cP_X(ij)\ne B$; there are at most $\theta k^2$ such pairs and at most $(1+\xi)\theta pu^2k^2$ such edges;
		\item for pairs $ij\notin X$ with $\cP_X(ij)=B$, the edges receiving a colour outside $B$, of which there are in total at most $\ell\alpha pu^2\binom k2$.
	\end{enumerate}
	Therefore
	\[
	|\{e\in E(G):\chi(e)\notin B\}|
	\le
	\left(\sigma+(1+\xi)\theta+\frac{\ell\alpha}{2}\right)pn^2.
	\]
	By \eqref{eq:stability-distance-errors} and part~(1) of \cref{lem:host}, the right-hand side is smaller than
	$\eps e(G)$ for all sufficiently large $n$.  Hence, if
	$d_{e(H)-1}(\chi)\ge\eps e(G)$, none of the regularity data representing $\chi$ can satisfy \eqref{eq:paletteclose}.  The contrapositive form of \cref{prop:hall-stability}, namely \eqref{eq:entropydeficit}, therefore applies to the partial template in every such tuple.
	
	Let $\mathcal S_{\mathrm{far}}$ be the collection of regularity data whose partial templates do not satisfy \eqref{eq:paletteclose} for any $B\in\binom{[\ell]}{e(H)-1}$, and let
	\[
	N_{\mathrm{far}}
	:=\bigl|\{\chi\in\cR_{H,\ell}(G):d_{e(H)-1}(\chi)\ge\eps e(G)\}\bigr|.
	\]
	If $\mathcal S_{\mathrm{far}}=\emptyset$, then $N_{\mathrm{far}}=0$ and there is nothing to prove.  Hence assume that $\mathcal S_{\mathrm{far}}$ is nonempty.
	Every colouring counted by $N_{\mathrm{far}}$ is represented by at least one element of
	$\mathcal S_{\mathrm{far}}$, and hence
	\[
	N_{\mathrm{far}}\le\sum_{\mathcal D\in\mathcal S_{\mathrm{far}}}N(\mathcal D).
	\]
	For each $\mathcal D\in\mathcal S_{\mathrm{far}}$,
	\begin{align*}
		&(1+\xi)pu_{\mathcal D}^{2}
		\Ent_{X_{\mathcal D}}(\cP_{\mathcal D})\\
		&\qquad\le
		(1+\xi)pu_{\mathcal D}^{2}
		\left(\binom{k_{\mathcal D}}2\log(e(H)-1)-\kappa k_{\mathcal D}^{2}\right)\\
		&\qquad\le
		(1+\xi)\frac{pn^2}{2}\log(e(H)-1)
		-(1+\xi)\kappa(1-\rho)^2pn^2,
	\end{align*}
	where we used $k_{\mathcal D}u_{\mathcal D}=n-|V_0|\ge(1-\rho)n$.
	Applying \eqref{eq:data-count} to $\mathcal S_{\mathrm{far}}$, using \eqref{eq:stability-entropy-errors}, and then decreasing a positive constant $\tau$ if necessary, we obtain
	\[
	\log N_{\mathrm{far}}
	\le \frac{1-2\xi}{2}pn^2\log(e(H)-1)-\tau pn^2
	\]
	for some $\tau=\tau(H,\ell,\eps)>0$.  For all sufficiently large $n$, part~(1) of \cref{lem:host} gives $e(G)\ge(1-2\xi)pn^2/2$, and hence
	\[
	\log N_{\mathrm{far}}
	\le e(G)\log(e(H)-1)-\tau pn^2.
	\]
	The corresponding upper bound $e(G)\le(1+2\xi)pn^2/2$ now gives
	\[
	N_{\mathrm{far}}
	\le(e(H)-1)^{(1-\gamma)e(G)}
	\]
	for some $\gamma=\gamma(H,\ell,\eps)>0$.  Since the constant $(e(H)-1)$-palette gives $R_{H,\ell}(G)\ge(e(H)-1)^{e(G)}$ and $e(G)\to\infty$ with high probability in this range, it also follows that $N_{\mathrm{far}}=o(R_{H,\ell}(G))$.  This proves \cref{thm:stability}.
	
	\section{Deletion bounds and the many-colour regime}
	
	The palette product bound in \cref{lem:hallgeneral,prop:localthreshold} determines the entropy--Tur\'an density throughout the Hall range.  Beyond this interval, we derive upper and lower bounds in terms of the edge-deletion density profile.  These estimates determine the leading term as $\ell\to\infty$ and identify exactly when the template with one common $(e(H)-1)$-palette attains the optimal limiting entropy density for every fixed number of colours.
	
	\subsection{Upper bound for the entropy--Tur\'an density}\label{subsec:Hminus-upper}
	
	We next prove the upper estimate \eqref{eq:lambda-upper-minus} in \cref{thm:deletion-bounds}.  Let $\cP$ be an $H$-SDR-free $\ell$-palette template on $K_n$, and let $L=L(\cP)$ be the graph formed by the edges whose palettes have size at least $e(H)$.
	
	The following observation separates the edges whose palettes have size at least $e(H)$ from the remaining edges, each of which contributes at most $\log(e(H)-1)$ to the entropy.  Hall's theorem shows that the former edges cannot contain any one-edge deletion of $H$.
	
	\begin{lemma}\label{lem:large-Hminus-free}
		The graph $L$ is $\mathcal H^-$-free.
	\end{lemma}
	
	\begin{proof}
		Suppose that $L$ contains a copy of $H-e_0$ for some $e_0\in E(H)$.  Add the missing edge $e_0$ in the ambient complete graph.  On the other $e(H)-1$ edges, every palette has size at least $e(H)$, while the palette on $e_0$ is nonempty.  We verify Hall's condition for these $e(H)$ palettes.  A subfamily consisting only of the palette on $e_0$ has union of size at least one.  Every other nonempty subfamily contains a palette of size at least $e(H)$, so its union has size at least $e(H)$, which is at least the size of the subfamily.  Thus the palettes have an SDR, contradicting the assumption that the template is $H$-SDR-free.
	\end{proof}
	
	By \cref{lem:large-Hminus-free} and the Erd\H{o}s--Stone--Simonovits theorem for the finite family $\mathcal H^-$ \cite{ErdosStone,Simonovits},
	\[
	e(L)\le\bigl(\pi_-(H)+o(1)\bigr)\binom n2.
	\]
	Every edge outside $L$ has a palette of size at most $e(H)-1$.  Consequently,
	\[
	\sum_{e\in E(K_n)}\log|\cP(e)|
	\le e(L)\log \ell+
	\left(\binom n2-e(L)\right)\log(e(H)-1),
	\]
	and hence
	\[
	\lambda(H,\ell)
	\le \pi_-(H)\log \ell+\bigl(1-\pi_-(H)\bigr)\log(e(H)-1).
	\]
	This proves \eqref{eq:lambda-upper-minus}.
	
	If there is some $e\in E(H)$ such that $H-e$ is bipartite, then $\pi_-(H)=0$.  The preceding inequality and the constant template with a common $(e(H)-1)$-set give
	\[
	\lambda(H,\ell)=\log(e(H)-1)
	\]
	for every fixed $\ell\ge e(H)$.  Together with \cref{thm:lambda-transfer}, this also shows that, for every fixed $\ell$, the above-threshold per-edge exponential base is $e(H)-1$.  More precisely, whenever
	$p n^{1/m_2(H)}\to\infty$,
	\[
	\frac{\log R_{H,\ell}(G(n,p))}{e(G(n,p))}
	\longrightarrow\log(e(H)-1)
	\]
	in probability.  The matching lower bound is given by all colourings using a fixed set of $e(H)-1$ colours.
	
	\subsection{Bounds from the edge-deletion density profile}
	
	Fix $1\le j\le e(H)-1$.  Recall that
	\[
	\mathcal H_j^-:=\{H-S:S\subseteq E(H),\ |S|=j\}.
	\]
	By the Erd\H{o}s--Stone--Simonovits theorem, the Tur\'an density of this finite family is
	$\pi_j(H)$.  Let $L_n$ be an $\mathcal H_j^-$-free graph with the maximum possible number of edges.  Then
	\[
	e(L_n)=\bigl(\pi_j(H)+o(1)\bigr)\binom n2.
	\]
	Assign the palette $[\ell]$ to every edge of $L_n$, and assign one fixed $j$-set of colours to each edge outside $L_n$.  Every copy of $H$ has at least $j+1$ edges outside $L_n$: otherwise one may enlarge the set of outside edges to a $j$-set $S\subseteq E(H)$, producing a copy of $H-S$ inside $L_n$, a contradiction with $L_n$ being $\mathcal H_j^-$-free.  The $j+1$ outside palettes have union of size $j$, so Hall's condition fails.  The template is therefore $H$-SDR-free and has entropy per edge
	\[
	\pi_j(H)\log \ell+\bigl(1-\pi_j(H)\bigr)\log j+o(1).
	\]
	Taking the maximum over $j$ proves \eqref{eq:lambda-lower-profile}.
	
	Together with the upper bound for $\lambda(H,\ell)$ from \cref{subsec:Hminus-upper}, this gives
	\[
	\pi_-(H)\log \ell
	\le\lambda(H,\ell)
	\le\pi_-(H)\log \ell+\bigl(1-\pi_-(H)\bigr)\log(e(H)-1),
	\]
	in addition to the lower bound $\lambda(H,\ell)\ge\log(e(H)-1)$ supplied by the common $(e(H)-1)$-palette.  Dividing by $\log \ell$ proves the limit in \cref{thm:deletion-bounds}.  If $\pi_-(H)=0$, then some member of $\mathcal H^-$ is bipartite and the upper bound gives $\lambda(H,\ell)=\log(e(H)-1)$ for every fixed $\ell$.  If $\pi_-(H)>0$, the $j=1$ lower bound exceeds $\log(e(H)-1)$ for all sufficiently large $\ell$.  This proves the final assertion of \cref{thm:deletion-bounds}.
	
	For comparison, the same construction can be applied directly to the random host: assign $[\ell]$ to the edges of $G(n,p)\cap L_n$ and the fixed $j$-set to all remaining edges of $G(n,p)$.  Chernoff's inequality applies to $e(G(n,p))$ and, when $\pi_j(H)>0$, to $e(G(n,p)\cap L_n)$.  When $\pi_j(H)=0$, the latter variable has expectation $o(pn^2)$, and Markov's inequality gives $e(G(n,p)\cap L_n)=o(e(G(n,p)))$ with high probability.  It follows that, for every fixed $j\in[e(H)-1]$, whenever $pn^2\to\infty$, with high probability
	\[
	R_{H,\ell}(G(n,p))
	\ge
	\exp\left(\bigl(\pi_j(H)\log\ell+(1-\pi_j(H))\log j-o(1)\bigr)e(G(n,p))\right).
	\]
	In the above-threshold range this inequality also follows from \cref{thm:lambda-transfer} and \eqref{eq:lambda-lower-profile}.
	
	Specialising the complete-graph construction to one deleted edge gives a useful explicit comparison.  Put $d=\min_{e\in E(H)}\chi(H-e)$.  If $d\ge3$, then $\lambda(H,\ell)\ge\pi_-(H)\log\ell$, so the $(e(H)-1)$-colour rate fails as soon as
	\[
	\ell>(e(H)-1)^{1/\pi_-(H)}
	=(e(H)-1)^{(d-1)/(d-2)}.
	\]
	If $d=2$, then $H-e$ is bipartite for some edge, and
	\cref{thm:deletion-bounds,thm:lambda-transfer} show that $\lambda(H,\ell)=\log(e(H)-1)$ for every fixed $\ell\ge e(H)$.
	
	\subsection{Further questions on \texorpdfstring{$\lambda(H,\ell)$}{lambda(H,l)}}\label{subsec:further-lambda}
	
	The results above determine $\lambda(H,\ell)$ throughout the Hall range and determine its leading term as $\ell\to\infty$, but leave its value open for many intermediate pairs $(H,\ell)$.  The templates constructed above from $\mathcal H_j^-$-free graphs give the explicit lower bound in \eqref{eq:lambda-lower-profile}; the remaining question is whether other palette structures can give a larger density.
	
	\begin{problem}\label{prob:determine-lambda}
		Determine $\lambda(H,\ell)$ for fixed $H$ and $\ell\ge e(H)$.  In particular, decide whether
		\[
		\lambda(H,\ell)=
		\max_{1\le j\le e(H)-1}
		\bigl(\pi_j(H)\log\ell+(1-\pi_j(H))\log j\bigr),
		\]
		or at least whether
		\[
		\lambda(H,\ell)=\log(e(H)-1)
		\quad\Longleftrightarrow\quad
		\ell^{\pi_j(H)}j^{1-\pi_j(H)}\le e(H)-1
		\quad\text{for every }1\le j\le e(H)-1.
		\]
	\end{problem}
	
	This is a rainbow-specific instance of the broader problem of describing entropy densities and extremal templates for multicolour hereditary properties \cite[Problem~5.1]{FROU}.  Generalised Erd\H{o}s--Rothschild problems can have competing multipartite and iterated extremal templates \cite{FROU,GuptaPehovaPowierskiStaden}, so the equality in \cref{prob:determine-lambda} is not a formal consequence of the deletion bounds.
	
	The case $H=K_4$ already contains a concrete unresolved interval.  Here $e(K_4)=6$, and the Hall range gives
	\[
	\lambda(K_4,\ell)=\log5
	\qquad (6\le\ell\le11),
	\]
	whereas deleting one edge and using a balanced bipartite graph gives
	\[
	\lambda(K_4,\ell)\ge\frac12\log\ell>\log5
	\qquad (\ell\ge26).
	\]
	For $\ell\ge12$, the result of H\`an, Hoppen, M\"uller and Schmidt also gives the upper bound $\lambda(K_4,\ell)\le\frac23\log\ell$ \cite{HanHoppenMullerSchmidt}.  Thus the present bounds leave open whether $\lambda(K_4,\ell)=\log5$ for $12\le\ell\le25$.  More generally, for $H=K_t$ and $1\le s\le t-2$, let
	\[
	d_s(t):=
	\min_{\substack{a_1+\cdots+a_s=t\\a_1,\ldots,a_s\ge1}}
	\sum_{i=1}^s\binom{a_i}{2}.
	\]
	Let $T_s(n)$ be the balanced complete $s$-partite graph.  Every set of $t$ vertices determines at least $d_s(t)$ pairs lying inside the vertex classes.  Assigning the palette $[\ell]$ to the cross-pairs of $T_s(n)$ and one fixed $(d_s(t)-1)$-set of colours to the pairs inside its classes therefore gives a deletion template with per-edge exponential base
	\[
	\ell^{1-1/s}\bigl(d_s(t)-1\bigr)^{1/s},
	\]
	which supplies explicit test cases for \cref{prob:determine-lambda}.
	
\section*{Acknowledgement}
The authors gratefully acknowledge the Fourth ECOPRO Student Research
Program, held at the Institute for Basic Science (IBS) in summer 2026,
for its support. The authors acknowledge the use of AI tools during the
exploratory stage of this project. All mathematical arguments and proofs
presented in the final manuscript were developed and rigorously verified
by the authors. The authors take full responsibility for the content of
the manuscript.

\end{document}